\documentclass[11pt,a4paper]{article}
\usepackage[T1]{fontenc}
\usepackage{lmodern}
\usepackage{amsmath,amssymb,amsthm,mathtools}
\usepackage{booktabs,array}
\usepackage{tikz}
\usetikzlibrary{arrows.meta,positioning,fit}
\usepackage[margin=27mm]{geometry}
\usepackage{microtype}
\usepackage[hidelinks,hypertexnames=false]{hyperref}
\hypersetup{pdftitle={Cayley-graph density of Thompson's group F: local deletions and finite-window bounds},pdfauthor={Thomas Prellberg}}
\usepackage{enumitem,needspace,float}
\setlist{itemsep=2pt,topsep=4pt}
\newtheorem{theorem}{Theorem}[section]
\newtheorem{lemma}[theorem]{Lemma}
\newtheorem{proposition}[theorem]{Proposition}

\theoremstyle{definition}

\theoremstyle{remark}

\newtheorem{example}{Example}[section]
\newcommand{\dens}{\operatorname{dens}}
\newcommand{\htree}{\operatorname{ht}}
\newcommand{\wt}{\operatorname{wt}}
\newcommand{\BB}{\operatorname{BB}}
\newcommand{\Q}{\mathbb Q}
\newcommand{\E}{\mathrm E}
\newcommand{\HH}{\mathrm H}
\newcommand{\LL}{\mathrm L}
\newcommand{\CC}{\mathrm C}
\newcommand{\JJ}{\mathrm J}
\newcommand{\ind}{\mathbf1}
\newcommand{\Gh}{\Gamma_{\mathrm h}}
\DeclareMathOperator{\pref}{pref}
\DeclareMathOperator{\state}{state}
\DeclareMathOperator{\cat}{cat}
\title{Cayley-graph density of Thompson's group $F$:\\local deletions and finite-window bounds}
\author{Thomas Prellberg\\
\small School of Mathematical Sciences, Queen Mary University of London\\
\small Mile End Road, London E1 4NS, United Kingdom\\
\small \texttt{t.prellberg@qmul.ac.uk}}
\date{}
\begin{document}
\maketitle
\begin{abstract}
Let $\Gamma$ be the Cayley graph of Thompson's group $F$ with its standard generators. Belk and Brown constructed finite marked-forest subgraphs of limiting density $7/2$, and Guba showed that deleting certain low-degree vertices gives density greater than $3.5004$. We refine this approach in two directions. First, a deterministic interval-deletion rule followed by root-sensitive triple and pair deletions yields finite induced subgraphs with
\[
\dens(\Gamma)>3.50074529.
\]
The interval rule is evaluated by a nine-state recurrence, while the simultaneous deletion conditions created by split and merge operations are computed exactly by a common-suffix first-passage argument. Second, we determine the joint distribution of the categories of a tree and its two children and reduce optimization over every fixed category-window retention rule to a weighted densest-subgraph problem. An explicit edge-allocation certificate shows that, for every fixed $a\ge0$, the optimal limiting density of a rule depending on the window $[-a,1]$ is exactly $7/2$. Thus arbitrary finite left context together with one right-hand category does not improve the Belk--Brown limit, whereas the root-sensitive whole-segment construction does. All constants are explicit elements of $\Q(\sqrt3,\sqrt{2\sqrt3-1})$.
\end{abstract}

\section{Introduction and the density bound}\label{sec:intro}
Let $\Gamma$ be the undirected Cayley graph of Thompson's group $F$ with its standard generating set $\{x_0,x_1\}$. For a nonempty finite vertex set $Y$, define
\[
\dens(Y)=\frac{2|E(Y)|}{|Y|},\qquad
\dens(\Gamma)=\sup_{\varnothing\ne Y\subset F\text{ finite}}\dens(Y),
\]
where $E(Y)$ denotes the undirected edges with both endpoints in $Y$. If $\partial_EY$ is the undirected edge boundary, then
\begin{equation}\label{eq:cheeger}
\dens(Y)=4-\frac{|\partial_EY|}{|Y|},\qquad
\dens(\Gamma)=4-h_E(\Gamma).
\end{equation}
Here $h_E(\Gamma)=\inf_{\varnothing\ne Y\subset F\text{ finite}}|\partial_EY|/|Y|$ is the edge-isoperimetric constant.
The relevance of \eqref{eq:cheeger} is the long-standing open problem of whether Thompson's group $F$ is amenable: for the standard two-generator Cayley graph, amenability is equivalent to $\dens(\Gamma)=4$. The forest-diagram description of $F$ is due to Belk and Brown \cite{BB}; general background is given in \cite{CFP}. Belk and Brown's finite marked-forest sets have densities tending to $7/2$, and Guba's deletion of low-degree triples gave the first strict improvement, proving $\dens(\Gamma)>3.5004$ \cite{GubaDensity}; related systematic use of deletion to increase Cayley-graph density was developed in \cite{AGLP}. The present paper continues this programme by sharpening the deletion analysis and raising the explicit lower bound further; this moves the known density towards the amenable value $4$, but a lower bound strictly below $4$ does not by itself decide amenability.

The contribution has two parts. First, we make Guba's deletion strategy explicit for a larger root-sensitive family. The initial deletions are selected by a deterministic weighted-interval recurrence whose values and selected positions are defined together. A split or merge changes the category word of the forest, so verifying that several neighbouring vertices have been deleted requires applying the same recurrence to several related words. Lemma~\ref{lem:common-suffix} gives the exact comparison once the unread suffixes agree, and the resulting frequencies are algebraic. Section~\ref{sec:horizontal} develops the recurrence from concrete examples, and Section~\ref{sec:gadgets} connects each deletion test to an actual Cayley-graph neighbour.

Second, we isolate a limitation of fixed-window category rules. The joint distribution of the category of a tree and the categories of its two children turns every retention rule on a prescribed finite window into a weighted densest-subgraph problem. Theorem~\ref{thm:one-right} gives an explicit upper-bound certificate showing that arbitrary finite left context, together with only the immediate right category, cannot improve the limiting density $7/2$. This separates that one-sided fixed-window class from the root-sensitive whole-segment constructions used for the improved lower bound.

The resulting explicit density bound is the following.
\begin{theorem}\label{thm:main}
There are finite induced subgraphs of $\Gamma$ with density greater than $3.50074529$. More precisely, the construction below gives an explicit algebraic number $D_+$ such that
\[
\dens(\Gamma)\ge D_+,\qquad
3.5007452936557<D_+<3.5007452936559.
\]
In particular, $h_E(\Gamma)<0.49925471$.
\end{theorem}

The construction first selects intervals of marker positions to maximize an additive score, using specified tie conventions, and then deletes disjoint triples and pairs whose split and merge neighbours satisfy explicit deletion tests. We denote its horizontal score and deleted mass by $\omega$ and $\rho$, the combined frequencies of three ordinary triple families and two pair families by $q_3$ and $q_2$, and the frequency of a separated triple by $p_0$. In this notation
\begin{equation}\label{eq:Dmain}
D_+=\frac72+
\frac{\omega+q_3+p_0+2q_2}{2(1-\rho-3q_3-3p_0-2q_2)}.
\end{equation}
Each quantity is evaluated by closed formulas in Sections~\ref{sec:horizontal}--\ref{sec:counts}; the explicit algebraic coefficients are collected in \eqref{eq:mainconstants}. The interval rule is optimal for its score objective, not asserted to optimize the density ratio. Likewise, $D_+$ is a limiting \emph{lower-bound expression}, because simultaneous local deletions can share incident edges. The horizontal stage alone has an exactly evaluated limiting density.

Section~\ref{sec:BB} establishes the local forest measure. Section~\ref{sec:horizontal} derives the prefix and first-passage formulas. Sections~\ref{sec:gadgets}--\ref{sec:counts} give the geometric deletions and prove Theorem~\ref{thm:main}. Section~\ref{sec:windows} proves the finite-window optimization theorem and its explicit one-sided application. All enumerations used in the proofs are derived from the displayed recurrences and tables.

\section{Marked forests and their local measure}\label{sec:BB}
\subsection{The partial Cayley action}
A rooted ordered binary tree is the one-leaf tree $E$, or an ordered pair $(T_0,T_1)$ of such trees. Write $|T|$ for its number of leaves and
\[
\htree(E)=0,\qquad \htree((T_0,T_1))=1+\max(\htree(T_0),\htree(T_1)).
\]
For a finite sequence of trees $\mathcal F=(T_1,\ldots,T_m)$, write $v(\mathcal F,i)$ for the marked forest whose marker is at $T_i$. A marked forest is one such pair of an underlying forest and a marker position. The set $\BB(n,k)$ consists of all marked forests with $n$ leaves and all tree heights at most $k$.

We use the realization of these forests as a finite Cayley subgraph described in \cite[Sections 2 and 4]{GubaDensity}, based on \cite{BB}. The four partial operations move the marker left, move it right, split the marked nontrivial root and mark its left child, and merge the marked tree with its right neighbour and mark the new root. The last operation remains in $\BB(n,k)$ exactly when both input heights are at most $k-1$. Each split is inverse to a merge. We use undirected adjacency, so inversion of the edge-labelling convention has no effect on any count below. Each marked forest is one vertex; different marker positions are different vertices.

For $k\ge3$, distinguish the categories
\[
\E:\ E,\qquad
\HH:\ \htree(T)=k,\qquad
\LL:\ 1\le\htree(T)\le k-1.
\]
Two subcategories of $\LL$ will be used: the single tree $C=(E,E)$, denoted $\CC$, and the trees $J$ of exact height $k-1$, denoted $\JJ$. They are disjoint for $k\ge3$.

Write $\cat(T)\in\{\E,\LL,\HH\}$ for the category of a tree $T$. The \emph{category word} of $(T_1,\ldots,T_m)$ is $\cat(T_1)\cdots\cat(T_m)$. A \emph{segment} is a maximal consecutive block of positions whose categories lie in $\{\E,\HH\}$. If a segment meets neither end of the forest, its two adjacent trees have category $\LL$; we call these its left and right \emph{guards}. Write $\varepsilon$ for the empty word. Thus $\LL W\LL$ denotes a segment $W\in\{\E,\HH\}^*$ together with its two bounding trees, allowing $W=\varepsilon$ for adjacent guards; the empty segment contributes no marker positions. A pattern of category symbols describes all forests with those categories, not a single choice of trees. We distinguish the category symbol $\HH$ from a specific tree $H$ of that category, and similarly for $\LL$.

\subsection{Coefficient ratios}
The leaf generating polynomial for trees of height at most $k$ satisfies
\begin{equation}\label{eq:phi}
\Phi_0(z)=z,\qquad \Phi_k(z)=z+\Phi_{k-1}(z)^2.
\end{equation}
Let $\xi_k>0$ be the unique solution of $\Phi_k(\xi_k)=1$. Then $\xi_k\downarrow1/4$: at $z=1/4$ the iteration stays below $1/2$, while at every $z>1/4$ it eventually exceeds one. The generating functions for unmarked and marked forests are respectively
\begin{equation}\label{eq:forestGF}
\frac1{1-\Phi_k(z)}=\sum_n\alpha_k(n)z^n,\qquad
\frac{\Phi_k(z)}{(1-\Phi_k(z))^2}=\sum_n\beta_k(n)z^n.
\end{equation}
Thus $|\BB(n,k)|=\beta_k(n)$. For $k\ge1$ the positive coefficients of both $z$ and $z^2$ in $\Phi_k$ imply that $\xi_k$ is its unique solution of $\Phi_k(z)=1$ of smallest modulus; the zero of $1-\Phi_k$ is simple.

\begin{lemma}[Local coefficient ratio]\label{lem:ratio}
Fix $k\ge1$. If $R(z)$ is analytic in a neighbourhood of $|z|\le\xi_k$, then
\begin{equation}\label{eq:ratio}
\frac{[z^n]R(z)(1-\Phi_k(z))^{-2}}{\beta_k(n)}\longrightarrow R(\xi_k).
\end{equation}
In particular, a fixed window of trees about a uniformly chosen marker converges to the product distribution assigning probability $\xi_k^{|T|}$ to a tree $T$.
\end{lemma}
\begin{proof}
The numerator and denominator have the same unique dominant double pole, and their leading Laurent coefficients differ by the factor $R(\xi_k)$, since $\Phi_k(\xi_k)=1$. Comparing coefficients proves \eqref{eq:ratio}; this is the elementary meromorphic coefficient estimate in \cite[Theorem IV.10]{FS}. For a prescribed window $T_{-a},\ldots,T_b$ the numerator is $z^{\sum_i|T_i|}$. Its value at $\xi_k$ is the product of the stated probabilities. Missing windows at forest ends have only a simple pole and zero limiting proportion.
\end{proof}

At the pole, the category probabilities are
\begin{equation}\label{eq:finiteweights}
e_k=\xi_k,\quad h_k=1-\sqrt{1-\xi_k},\quad
\ell_k=\sqrt{1-\xi_k}-\xi_k.
\end{equation}
Their limits are
\begin{equation}\label{eq:weights}
e=\frac14,\qquad h=1-\frac{\sqrt3}2,\qquad
\ell=\frac{\sqrt3}2-\frac14=1-e-h.
\end{equation}
The weight of $C$ tends to $e^2$. Reversing the recurrence \eqref{eq:phi} once more gives
\begin{equation}\label{eq:j}
\Phi_{k-1}(\xi_k)-\Phi_{k-2}(\xi_k)\longrightarrow
j=\frac{\sqrt3-\sqrt{2\sqrt3-1}}2.
\end{equation}
This is the total weight of $\JJ$. Summing over all trees $J$ of exact height $k-1$, each of the ordered root forms $(J,E)$ and $(E,J)$ has limiting weight $ej$.

\begin{lemma}[Segment rewards]\label{lem:segments}
Fix a reward depending on the category word of a segment. Sum it once over each segment of each underlying forest represented in $\BB(n,k)$, and divide by $|\BB(n,k)|$. If the absolute reward is bounded by a constant times the segment length, this normalized sum converges to the corresponding mean reward per position in the product category process. This holds first at fixed $k$ and then as $k\to\infty$. The same assertion holds for the finitely many root-pattern conditions used below.
\end{lemma}
\begin{proof}
At fixed $k$ the weight of a word of length $m$ over $\E,\HH$, summed over all such words, is $(e_k+h_k)^m$. The two $\LL$ guards multiply this by $\ell_k^2$. The tails of a length-weighted reward are therefore summable. One can also apply Lemma~\ref{lem:ratio} directly: replacing each letter weight by $z$ or $\Phi_k(z)-\Phi_{k-1}(z)$ gives a rational finite-state series, analytic at and near $|z|\le\xi_k$, because its absolute transition row sums there are at most $e_k+h_k<1$. More explicitly, the absolute tail is bounded by a constant times the series $\sum_{m>M}m(z+\Phi_k(z)-\Phi_{k-1}(z))^m$, multiplied by the guard polynomials. This series is analytic on the same neighbourhood. Lemma~\ref{lem:ratio} bounds the limiting tail by its value at $\xi_k$, which tends to zero as $M\to\infty$. First truncate segment lengths, apply the fixed-window result, and then let $M\to\infty$. Since $\ell_k\to\ell>0$, the argument is uniform for all sufficiently large $k$. Specifying a finite root pattern contributes a product or difference of tree generating polynomials and does not change the tail estimate. Segments meeting a forest end give only a simple-pole contribution, hence a vanishing fraction of marked vertices.
\end{proof}

A horizontal move is absent only at the corresponding forest end. A split is absent precisely at a trivial marked tree. For a finite group subset, the number of missing moves by a generator equals the number by its inverse. Since $\alpha_k(n)/\beta_k(n)\to0$, Lemma~\ref{lem:ratio} consequently gives
\begin{equation}\label{eq:base}
\lim_{n\to\infty}\dens(\BB(n,k))=4-2\xi_k,
\qquad \lim_{k\to\infty}(4-2\xi_k)=\frac72.
\end{equation}

\section{Interval deletion and exact enumeration}\label{sec:horizontal}
We first specify which vertices are deleted, and then count them. Throughout this section, a word records the categories of consecutive trees in a fixed forest. Selecting a position in that word means deleting the vertex whose marker is at the corresponding tree; it does not mean removing the tree from the forest.

\subsection{The score of an interval}
Fix a segment with category word $W=b_1\cdots b_m\in\{\E,\HH\}^m$ between two $\LL$ guards. Number its positions from $1$ to $m$, and put $b_{m+1}=\LL$ for its right guard. Every marker in the segment has both horizontal neighbours. A marker at a height-$k$ tree also has its split neighbour, but no merge neighbour. A marker at $E$ has no split neighbour; its merge neighbour is available precisely when the next tree is not in $\HH$. Consequently its degree in the original graph $\BB(n,k)$ is
\[
3-\ind_{(b_i,b_{i+1})=(\E,\HH)}.
\]
For an interval $[u,v]$ of $d=v-u+1$ positions, let $m([u,v])$ count the $\E\HH$ transitions whose first position lies in the interval. This includes a transition from $v$ to $v+1$. The degree sum is $3d-m([u,v])$, and there are $d-1$ internal horizontal edges. There are no internal vertical edges, because a split or merge changes the underlying forest. The interval is therefore incident to $2d+1-m([u,v])$ undirected edges.

For a set $X$ of $d$ vertices incident to $c$ undirected edges, define its score relative to the benchmark $7/2$ by
\begin{equation}\label{eq:score}
W(X)=7d-4c.
\end{equation}
The coefficient $4$ accounts for the fact that density counts each undirected edge twice. Set
\begin{equation}\label{eq:localvalues}
a_i=\begin{cases}3,&(b_i,b_{i+1})=(\E,\HH),\\-1,&\text{otherwise}.
\end{cases}
\qquad W([u,v])=\sum_{i=u}^v a_i-4.
\end{equation}
In particular, $a_m=-1$. The last term in the interval score is a cost of four for starting an interval.

\subsection{A deterministic choice of deleted positions}\label{sec:deletion-rule}
For $0\le i\le m$, let $O_i$ be the maximum total score of a family of disjoint intervals in $[1,i]$, allowing the empty family. Let $I_i$ be the maximum among such families with an interval ending at $i$. The letter $I$ records this final interval, which may be extended at the next step; $O_i$ does not require position $i$ to be outside the chosen intervals. The recurrences are
\begin{equation}\label{eq:schedule}
I_i=\max(O_{i-1}+a_i-4,I_{i-1}+a_i),\qquad
O_i=\max(O_{i-1},I_i),
\end{equation}
with $O_0=0$ and $I_0=-\infty$. The first term starts a new final interval, and the second extends the previous final interval.

The numerical maxima alone do not specify which positions to delete. We therefore also define a candidate set $\mathcal C_i$ attaining $I_i$ and an accepted set $\mathcal D_i$ attaining $O_i$. Start with $\mathcal C_0=\mathcal D_0=\varnothing$, and use
\begin{align}\label{eq:selected-sets}
\mathcal C_i&=
\begin{cases}
\mathcal C_{i-1}\cup\{i\},&I_{i-1}\ge O_{i-1}-4,\\
\mathcal D_{i-1}\cup\{i\},&I_{i-1}<O_{i-1}-4,
\end{cases}\notag\\
\mathcal D_i&=
\begin{cases}
\mathcal C_i,&I_i\ge O_{i-1},\\
\mathcal D_{i-1},&I_i<O_{i-1}.
\end{cases}
\end{align}
Thus a tie in the first maximum continues the final interval, and a tie in the second accepts the candidate. We call a position \emph{accepted} once it belongs to $\mathcal D_i$. For the complete word $W$, the deletion set is
\[
\mathcal D(W)=\mathcal D_m.
\]
The connected components of this set are the selected intervals. Apply this rule separately to every segment with two $\LL$ guards. All $\LL$ positions, and all positions in segments meeting a forest end, are retained. Let $\Gh$ denote the induced subgraph remaining after these deletions.

\begin{lemma}[Properties of the deletion rule]\label{lem:selection}
The rule maximizes the total interval score with the stated tie conventions. For every $i$, one has $\mathcal D_{i-1}\subseteq\mathcal D_i\subseteq\mathcal C_i$. Every selected interval begins and ends at a position with $a_i=3$, and distinct selected intervals are not adjacent. The final position of a segment, and an initial $\HH$ position, are retained.
\end{lemma}
\begin{proof}
A family counted by $I_i$ has a final interval starting at $i$, or a final interval extending an interval ending at $i-1$. These two cases give the first recurrence. For $O_i$, a maximizing family omits $i$ or has an interval ending there, giving the second recurrence. The set choices in \eqref{eq:selected-sets} realize these maxima.

The inclusion $\mathcal D_i\subseteq\mathcal C_i$ follows inductively. Starting a new candidate includes $\mathcal D_{i-1}$ explicitly, and extending the previous candidate includes it by the induction hypothesis. Accepting a candidate therefore never removes an already accepted position. An interval endpoint of value $-1$ could be removed to increase the score, and adjacent intervals could be joined to save a start cost of four. Such endpoints and adjacent intervals cannot occur in a maximizing family. Since the last value is $-1$, the final position is retained. An initial $\HH$ also has value $-1$ and cannot be covered by an interval starting earlier in the segment.
\end{proof}

\begin{example}[From the recurrence to three Cayley vertices]\label{ex:basic-deletion}
Take $W=\E\HH\E\HH$. The values are $(3,-1,3,-1)$, and the complete calculation is
\begin{center}
\begin{tabular}{ccccccc}
\toprule
$i$&$b_i$&$a_i$&$I_i$&$O_i$&$\mathcal C_i$&$\mathcal D_i$\\
\midrule
0&---&---&$-\infty$&0&$\varnothing$&$\varnothing$\\
1&$\E$&3&$-1$&0&$\{1\}$&$\varnothing$\\
2&$\HH$&$-1$&$-2$&0&$\{1,2\}$&$\varnothing$\\
3&$\E$&3&1&1&$\{1,2,3\}$&$\{1,2,3\}$\\
4&$\HH$&$-1$&0&1&$\{1,2,3,4\}$&$\{1,2,3\}$\\
\bottomrule
\end{tabular}
\end{center}
Thus $\mathcal D(W)=\{1,2,3\}$, not $\{1,2,3,4\}$: the final candidate is not accepted.

For a concrete forest, take $k=3$, $C=(E,E)$, $J=(C,C)$ and $H=(J,J)$. The forest
\[
\mathcal F=(C,E,H,E,H,C)
\]
has the segment $W$ between its two $C\in\LL$ guards. The rule deletes precisely the three vertices $v(\mathcal F,2)$, $v(\mathcal F,3)$ and $v(\mathcal F,4)$ from this forest's horizontal path. Their degrees in $\BB(n,3)$ are $(2,3,2)$; their two internal edges give $c=2+3+2-2=5$ and score $7\cdot3-4\cdot5=1$. All six trees remain in each of these vertex descriptions. It is the marker position that distinguishes the three vertices.
\end{example}

\begin{example}[A tie changes the deletion set]\label{ex:zero-score}
For $W=\E\HH\E\E\HH$, the values are $(3,-1,-1,3,-1)$. The candidate scores $I_i$ are $-1,-2,-3,0,-1$, and all accepted scores $O_i$ are zero. At $i=4$, the equality $I_4=O_3$ causes the rule to accept $\{1,2,3,4\}$. Hence
\[
\mathcal D(\E\HH\E\E\HH)=\{1,2,3,4\},
\qquad W([1,4])=0.
\]
Selecting only intervals of strictly positive score would therefore define a different deletion rule. This zero-score example will account for a deleted split neighbour in Section~\ref{sec:triples}.
\end{example}

\begin{proposition}[Exact horizontal accounting]\label{prop:horizontal}
No vertical edge has two horizontally deleted endpoints. More specifically, the split neighbour of every marker at a height-$k$ tree is horizontally retained. If $\rho$ is the limiting fraction of deleted vertices and $\omega$ is the limiting total interval score divided by the original number of vertices, then the horizontal induced subgraph has limiting density
\begin{equation}\label{eq:Dh}
D_{\mathrm h}=\frac72+\frac{\omega}{2(1-\rho)}.
\end{equation}
The limits are taken first as $n\to\infty$ and then as $k\to\infty$.
\end{proposition}
\begin{proof}
A deleted $E$ has no split neighbour. If followed by $E$, its merge neighbour has marker at $C\in\LL$ and is retained; if followed by $\HH$, no merge is available. It cannot be the final position before an $\LL$ guard.

A height-$k$ tree $H=(T_0,T_1)$ cannot merge within the height bound. If $T_0$ is nontrivial, splitting $H$ puts the marker at a tree in $\LL$, which is retained. If $T_0=E$, then $T_1$ has height $k-1$, so the split marker is at the last $E$ before an $\LL$ guard and is retained. This also proves the more specific assertion, whether or not the original marker at $H$ was deleted.

By Lemma~\ref{lem:selection}, the only edges with two deleted endpoints are the horizontal edges inside the selected intervals. Thus incident-edge counts add over all selected intervals. Summing \eqref{eq:score}, the loss of doubled edge count per original vertex tends to $(7\rho-\omega)/2$. Subtract this from the limiting density $7/2$ in \eqref{eq:base} and divide by the retained fraction $1-\rho$, obtaining \eqref{eq:Dh}.
\end{proof}

Score optimality is not density optimality. If a finite graph of density $D$ on $N$ vertices loses $r<N$ vertices incident to $c$ undirected edges, then, with $W=7r-4c$,
\begin{equation}\label{eq:ratio-objective}
D_{\rm new}-D=\frac{W+2(D-7/2)r}{2(N-r)}.
\end{equation}
The rule optimizes the additive score, not this ratio. Our lower bound needs only the specified construction and its exact enumeration.

\subsection{What the prefix state records}\label{sec:prefix-states}
For the enumeration, allow arbitrary $e,h>0$ with $e+h<1$, and put $\ell=1-e-h$. Letters $\E,\HH,\LL$ have probabilities $e,h,\ell$ in the product category process. For a finite word $P$ over $\{\E,\HH\}$, define
\[
\wt(P)=e^{\#\E(P)}h^{\#\HH(P)},\qquad \wt(\varepsilon)=1.
\]

A \emph{prefix} is the portion $P=b_1\cdots b_r$ of a segment read so far, beginning immediately after its left $\LL$ guard. For a nonempty prefix, scores have been processed only through position $r-1$: the score $a_r$ cannot yet be determined because $b_{r+1}$ has not been read. We call $b_r$ the \emph{pending letter}. The state records this letter and the clipped score difference
\[
\widehat d(P)=\max\{-4,I_{r-1}-O_{r-1}\}.
\]
For $r=1$ this uses $I_0=-\infty$ and $O_0=0$. Thus reading the first letter does not process any score.

Suppose a nonempty prefix ends in $b$, and the next letter read is $\sigma\in\{\E,\HH\}$. Put $a(b,\sigma)=3$ for $(b,\sigma)=(\E,\HH)$ and $a(b,\sigma)=-1$ otherwise. One step processes the score of $b$, with trial difference and accepted-score gain
\begin{equation}\label{eq:state-update}
x=\widehat d+a(b,\sigma),\qquad g=\max(0,x).
\end{equation}
The new pending letter is $\sigma$, and the new clipped difference is $\max(-4,x-g)$. This follows directly from \eqref{eq:schedule}, since $I_r-O_{r-1}=\max(-4,I_{r-1}-O_{r-1})+a_r$.

The resulting nine states and their transitions are shown in Table~\ref{tab:prefix}. Write $\state(P)$ for the row reached after reading $P$. The empty prefix has state 0 by convention; its two initial transitions simply read the first letter and have gain zero. For nonempty words, state 0 represents an actual pending $\HH$ with clipped difference $-4$. When the right guard $\LL$ is read, the pending score is $-1$ and the segment ends; this final step never increases $O_i$.

\begin{table}[htbp]
\centering
\caption{States after a prefix has been read but before its last letter has been processed. Each transition gives the next row and the increase in $O$. The entry $-4$ is a clipped value, representing every exact difference at most $-4$.}\label{tab:prefix}
\begin{tabular}{clcc}
\toprule
$i$&$(\widehat d,\text{pending letter})$&read $\E$&read $\HH$\\
\midrule
0&$(-4,\HH)$; also the empty prefix&$1;0$&$0;0$\\
1&$(-4,\E)$&$1;0$&$2;0$\\
2&$(-1,\HH)$&$3;0$&$4;0$\\
3&$(-2,\E)$&$5;0$&$6;1$\\
4&$(-2,\HH)$&$5;0$&$7;0$\\
5&$(-3,\E)$&$1;0$&$6;0$\\
6&$(0,\HH)$&$8;0$&$2;0$\\
7&$(-3,\HH)$&$1;0$&$0;0$\\
8&$(-1,\E)$&$3;0$&$6;2$\\
\bottomrule
\end{tabular}
\end{table}

For example, reading the successive prefixes of $\E\HH\E\HH$ gives states
\[
\state(\varepsilon)=0,\quad
\state(\E)=1,\quad
\state(\E\HH)=2,\quad
\state(\E\HH\E)=3,\quad
\state(\E\HH\E\HH)=6.
\]
The last transition has gain one: reading the fourth letter processes position 3 and accepts the three positions in Example~\ref{ex:basic-deletion}. In particular, $\state(P)$ always refers to the word $P$ itself, with its last letter still pending, not to the word after its right guard has been processed.

\subsection{Summing the prefix weights}\label{sec:prefix-weights}
Define
\[
G_i=\sum_{\substack{P\in\{\E,\HH\}^*\\\state(P)=i}}\wt(P).
\]
These are sums over all finite possible prefixes, not over prefixes of one fixed word, and not probabilities. The empty word contributes one to $G_0$. If $\delta_\sigma(i)$ denotes the next row in Table~\ref{tab:prefix}, decomposition by the last appended letter gives
\begin{equation}\label{eq:prefixbalance}
G_i=\ind_{i=0}+\sum_{\delta_\E(t)=i}eG_t+\sum_{\delta_\HH(t)=i}hG_t.
\end{equation}
Put
\begin{equation}\label{eq:Delta}
\Delta=1-eh(e^2+eh+h^2),\qquad c=1-eh(1+e+e^2).
\end{equation}
In particular $\Delta>0$, since $eh(e^2+eh+h^2)\le(e+h)^4/4<1$.

\begin{lemma}\label{lem:G}
The unique solution of \eqref{eq:prefixbalance} is
\begin{align}\label{eq:Gclosed}
G_2&=\frac{ehc}{\ell\Delta},&G_6&=\frac{e^2h^2(1+e+h)}{\ell\Delta},\notag\\
G_4&=hG_2,&G_7&=hG_4,&G_8&=eG_6,\\
G_3&=e(G_2+G_8),&G_5&=e(G_3+G_4),\notag\\
G_1&=G_2/h-G_6,&G_0&=\frac{1+hG_7}{1-h}.\notag
\end{align}
Moreover $\sum_iG_i=1/\ell$.
\end{lemma}
\begin{proof}
Reading the incoming transitions in Table~\ref{tab:prefix} gives
\[
\begin{array}{lll}
G_0=1+hG_0+hG_7,&G_1=e(G_0+G_1+G_5+G_7),&G_2=h(G_1+G_6),\\
G_3=e(G_2+G_8),&G_4=hG_2,&G_5=e(G_3+G_4),\\
G_6=h(G_3+G_5+G_8),&G_7=hG_4,&G_8=eG_6.
\end{array}
\]
Substitution verifies \eqref{eq:Gclosed}. The outgoing transition matrix has row sum $e+h<1$, so the linear system has a unique solution. Also, summing the weights of all words gives $\sum_iG_i=\sum_{m\ge0}(e+h)^m=1/\ell$.
\end{proof}

To count score per original marker, locate a position by its prefix since the preceding $\LL$ guard. A specified prefix has weight $\ell\wt(P)$, including that guard. Only states 3 and 8 give positive gain, and both require the next letter to be $\HH$. Summing these gains yields
\begin{equation}\label{eq:omegaformula}
\omega=\ell h(G_3+2G_8)
=\frac{e^2h^2\bigl(1+eh(1+2e+2h)+e^2h^2\bigr)}{\Delta}.
\end{equation}
There is no further right-guard factor here: after a gain has occurred, all possible continuations together have probability one. Summing the gains along a complete segment gives its final score $O_m$, so this counts exactly the total selected interval score.

\subsection{Deletion probability of a specified position}\label{sec:firstpassage}
The nine-state table determines future scores, but it does not record membership of a previously processed position in the candidate set. We now keep this additional information. Suppose that a prefix of length $\nu$ has been read, so position $\nu$ is pending, and fix a position $j<\nu$. By \eqref{eq:selected-sets}, there are three possibilities: $j$ is accepted if $j\in\mathcal D_{\nu-1}$; it is still a candidate if $j\in\mathcal C_{\nu-1}\setminus\mathcal D_{\nu-1}$; and it has been discarded if $j\notin\mathcal C_{\nu-1}$. An accepted position stays accepted. A discarded position cannot re-enter, because later steps add only later positions.

For a position still in the candidate, use the \emph{exact} difference
\[
d=I_{\nu-1}-O_{\nu-1}<0,
\]
not its clipped value. If $d<-4$, the next step starts a new interval and discards this position, so eventual acceptance is impossible. If $d\ge-4$, the next step continues the candidate. With next letter $\sigma$, the trial difference is $d+a(b_\nu,\sigma)$. A nonnegative trial difference accepts the candidate; a negative one is the new exact difference. In particular, at $d=-4$ continuation is still required. Clipping all differences below $-4$ to $-4$ would lose this distinction for an old candidate position, although it was harmless for the score table.

Let $r,u,v,w,x,y,z$ denote the probabilities of eventual acceptance of such a candidate position when the exact difference and pending letter are respectively
\[
(-4,\E),\ (-1,\HH),\ (-2,\E),\ (-2,\HH),\ (-3,\E),\ (-3,\HH),\ (-1,\E).
\]
The future letters are independent, with probabilities $e,h,\ell$, and the first future $\LL$ ends the segment. Acceptance before this termination is a first-passage event: the trial difference must reach a nonnegative value before the position is discarded. At $(-4,\HH)$ the probability is zero, since the next score is $-1$. Reading $\LL$ also causes failure for every unaccepted position.

Conditioning on the next letter gives
\begin{equation}\label{eq:passage}
\begin{array}{llll}
r=hu,&u=ev+hw,&v=ex+h,&w=ex+hy,\\
x=er+h,&y=er,&z=ev+h.&
\end{array}
\end{equation}
For instance, at $(-4,\E)$ a next $\E$ gives difference $-5$ and failure, whereas a next $\HH$ gives $(-1,\HH)$, proving $r=hu$. At $(-2,\E)$ a next $\HH$ gives trial difference $1$ and immediate acceptance, while a next $\E$ gives $(-3,\E)$; hence $v=ex+h$. These examples also explain the constant terms $h$ in the other equations. The coefficients of the unresolved transitions have row sums at most $e+h<1$, so this probability system has a unique solution. Solving its first six equations gives
\begin{equation}\label{eq:rclosed}
r=\frac{eh^2(1+e+h)}{\Delta},\qquad u=r/h,
\end{equation}
and the remaining probabilities follow from \eqref{eq:passage}.

We will also need sums of word weights rather than probabilities. If $R\in\{\E,\HH\}^*$ is the complete future word before the right guard, its probability is $\ell\wt(R)$. Thus for any specified acceptance event $\mathcal E$,
\begin{equation}\label{eq:suffix-normalization}
\Pr(\mathcal E)=\ell\sum_{R\in\{\E,\HH\}^*}\wt(R)\ind_{\mathcal E\text{ holds for }R}.
\end{equation}
This is the reason for the factors $1/\ell$ in the suffix counts in Section~\ref{sec:paircounts}.

Finally, let $P_i$ be the probability that the \emph{pending position itself} is eventually deleted, given a nonempty prefix in state $i$. This position has not yet been processed. At the next step it is included in the candidate whether that step starts a new interval or continues the old one. Consequently the clipped state is sufficient for this first step, after which the acceptance equations above apply. They give
\begin{equation}\label{eq:pvector}
(P_0,\ldots,P_8)=(0,r,u,v,w,x,t,y,z),\qquad t=ez+hu.
\end{equation}
For state 6, the next score is $-1$, giving $(-1,\E)$ or $(-1,\HH)$ according to the next letter, which explains $t$. State 0 gives zero. We also assign zero to the empty-prefix convention, which has no pending position.

Each potentially deleted marker has a unique nonempty prefix from its preceding guard. Multiplying its prefix weight by its conditional deletion probability and summing gives
\begin{equation}\label{eq:rhoformula}
\rho=\ell\sum_{i=0}^8G_iP_i.
\end{equation}
The empty word contributes zero, and markers in $\LL$ are never deleted. Lemma~\ref{lem:segments} transfers these product-process calculations to the finite marked-forest graphs. Substituting \eqref{eq:weights} into \eqref{eq:Gclosed}--\eqref{eq:rhoformula} gives the values of $\omega$ and $\rho$ in \eqref{eq:mainconstants}, and
\[
D_{\mathrm h}=3.50060007264292758570\ldots.
\]

\section{Local triples and root-sensitive pairs}\label{sec:gadgets}
We now select additional vertices for deletion from $\Gh$. All degrees and incident-edge counts in this section refer to $\Gh$, before any of these additional deletions are made. A \emph{target forest} is a forest containing the vertices selected at this second stage. A box around consecutive trees denotes the set of vertices obtained by placing the marker at each boxed position in that same forest. By contrast, an underlined position in a word specifies one marker whose membership in the first-stage deletion set $\mathcal D(W)$ is to be tested.

Unless otherwise stated, $P,Q,R,S$ are finite words over $\{\E,\HH\}$, possibly empty. Every displayed $\LL$ guard is an actual adjacent tree of category $\LL$; a displayed word is a consecutive part of a forest, and the rest of the forest is unrestricted. Each occurrence is counted once, by marking its specified centre $C$ or $J$. Its \emph{frequency} means the number of such occurrences divided by the original number of vertices in $\BB(n,k)$, in the successive limits $n\to\infty$ and $k\to\infty$. Thus a disjoint triple family of frequency $p$ deletes a fraction $3p$ of the original vertices.

We will repeatedly use a test for the $E$ immediately before the last $\HH$ in a segment $Q\E\HH$. Define
\begin{equation}\label{eq:AB}
\mathcal B=\{2,3,4,6,8\},\quad
B=\sum_{i\in\mathcal B}G_i=\frac{eh(1+e+h)}{\ell\Delta},\quad
A=\frac1\ell-B.
\end{equation}
Here $\state(Q)$ is evaluated after reading $Q$, before reading the displayed $\E\HH$. The two corresponding transitions in Table~\ref{tab:prefix} show that the displayed $E$ is accepted when its successor $\HH$ is read precisely when $\state(Q)\in\mathcal B$. In that case the new exact difference is zero; otherwise it is $-1$. The final $\HH$ is retained. The same test applies in $Q\E\HH Y$ with $Y\in\{\E,\HH\}$ followed by the right guard: the two remaining scores are $-1$, so they cause no new acceptance. The quantities $A$ and $B$ are therefore the sums of weights of prefixes failing and passing this test, respectively.

\subsection{Triples centred at \texorpdfstring{$C$}{C}}\label{sec:triples}
\begin{proposition}\label{prop:triples}
Select the boxed positions in each of the following eligible occurrences:
\begin{enumerate}[label=(\roman*)]
\item
$\LL\,Q\,\boxed{\E\ \HH\ \CC}\,\HH R\,\LL$, provided $\state(Q)\notin\mathcal B$;
\item
$\LL\,Q\,\E\,\boxed{\HH\ Y\ \CC}\,\HH R\,\LL$, where $Y\in\{\E,\HH\}$, provided $\state(Q)\in\mathcal B$.
\end{enumerate}
All three boxed vertices are present in $\Gh$. In left-to-right order their degrees there are $(2,3,2)$. They have two internal edges, hence five incident undirected edges and score one.
\end{proposition}
\begin{proof}
In (i), the left segment is $Q\E\HH$. By the preceding test its displayed $E$ is not accepted, and the final $\HH$ is retained. The $C$ marker is also retained, because $C\in\LL$. Splitting this $C$ into two leaves produces the segment
\[
Q\E\HH\E_1\E_2\HH R.
\]
The split neighbour has its marker at $E_1$. Immediately after reading $Q\E\HH$, the last $\HH$ is pending and the exact difference is $-1$. Reading $\E_1,\E_2,\HH$ processes the successive scores $-1,-1,3$, and gives differences $-2,-3,0$. The candidate is accepted at the last step, including the position $E_1$. Thus the split neighbour of the target $C$ is absent from $\Gh$.

The target $C$ cannot merge with its right-hand $\HH$, and its two horizontal neighbours are retained: the left one is the terminal $\HH$ of the left segment, and the right one is the initial $\HH$ of the right segment. Its degree is therefore two. The first boxed $E$ has no vertical neighbour, since it cannot split or merge with $\HH$. Its left neighbour is retained: an interval containing that neighbour but not this retained $E$ would have to end at a position of value $-1$. Hence this $E$ has degree two. The middle $\HH$ has both horizontal neighbours and its retained split neighbour, by Proposition~\ref{prop:horizontal}, so its degree is three.

In (ii), reading the $\HH$ immediately after the unboxed $E$ accepts that $E$ and leaves exact difference zero. The next two target scores, at $\HH,Y$ before the $C$ guard, are both $-1$. Thus the two boxed positions $\HH,Y$ are not accepted. After splitting $C$, the segment becomes
\[
Q\E\HH Y\E_1\E_2\HH R.
\]
Starting at the pending $\HH$ with difference zero, the scores at $\HH,Y,\E_1,\E_2$ are $-1,-1,-1,3$. The differences are $-1,-2,-3,0$. The acceptance tie at the final step includes $E_1$, so the split neighbour of $C$ has again been deleted.

The first boxed $\HH$ has lost its left horizontal neighbour, and retains its right horizontal neighbour and its split neighbour; its degree is two. The middle $Y$ retains both horizontal neighbours. Its third edge is its split edge when $Y\in\HH$, and its merge edge with $C$ when $Y=E$; in the latter case the merged tree $(E,C)$ has height two and is in $\LL$ because $k\ge3$. Finally, $C$ retains exactly its two horizontal edges, as in (i). The degree sum is seven and the two internal horizontal edges are counted twice in that sum, giving five incident edges.
\end{proof}

\begin{figure}[htbp]
\centering
\begin{tikzpicture}[x=0.88cm,y=0.84cm,>=Stealth,
 every node/.style={font=\small,inner sep=2.5pt}]
\foreach \x/\lab in {0/\LL,1/Q,2/\E,3/\HH,4/\E_1,5/\E_2,6/\HH,7/R,8/\LL}
  \node (s\x) at (\x,0) {$\lab$};
\node[draw,dashed,fit=(s4),inner sep=4pt] {};
\draw[->] (4.5,-.45)--(4.5,-1.35) node[midway,right] {merge};
\foreach \x/\lab in {0/\LL,1/Q,2/\E,3/\HH,4.5/\CC,6/\HH,7/R,8/\LL}
  \node at (\x,-1.8) {$\lab$};
\draw (1.68,-2.13) rectangle (4.85,-1.47);
\node[anchor=west] at (8.55,0) {deleted marker};
\node[anchor=west] at (8.55,-1.8) {three targets};
\end{tikzpicture}
\caption{An eligible $EHC$ triple. The dashed box marks the deleted split neighbour of the target $C$. The solid box denotes three different vertices, with the marker at $E$, at $H$, and at $C$ in the same target forest. The vertical arrow is the merge taking the dashed marker to the $C$ marker; it is not incident to all three targets.}\label{fig:triple}
\end{figure}
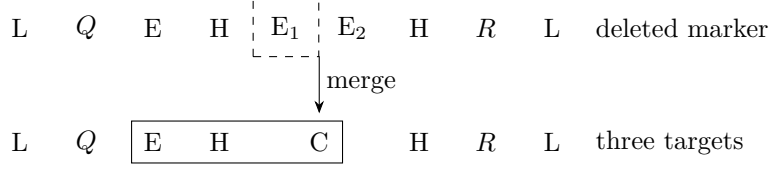

\Needspace{19\baselineskip}
\begin{example}[The split neighbour in an actual forest]\label{ex:triple}
Use the height-three tree $H$ from Example~\ref{ex:basic-deletion}, and take the target forest
\[
\mathcal F=(C,E,H,C,H,C).
\]
This is case (i) with $Q=R=\varepsilon$. Its two segments $\E\HH$ and $\HH$ have no deleted positions. The additional triple consists of the vertices marked at positions 2, 3 and 4.

The split neighbour of $v(\mathcal F,4)$ is
\[
v((C,E,H,E,E,H,C),4).
\]
Its segment word is $\E\HH\E_1\E_2\HH$. Example~\ref{ex:zero-score} gives $\mathcal D(\E\HH\E\E\HH)=\{1,2,3,4\}$, so the marker at $E_1$ (position 4 in the forest, position 3 in the segment) has indeed been deleted. The three target degrees in $\Gh$ are consequently $(2,3,2)$, and deleting the triple loses exactly five incident undirected edges when considered on its own. This example is why a zero-score acceptance in the first stage can be useful in the second stage.
\end{example}

For the frequencies, the two outer guards contribute $\ell^2$, each $C$ contributes $e^2$, and summing the unrestricted right word $R$ contributes
$\sum_R\wt(R)=1/\ell$. The allowed left words contribute $A$ in (i) and $B$ in (ii). Multiplying by the other displayed letter weights gives
\begin{equation}\label{eq:triplefreqs}
p_{EHC}=\ell e^3h^2A,\qquad
p_{HHC}=\ell e^3h^3B,\qquad
p_{HEC}=\ell e^4h^2B.
\end{equation}
For example, the first product is $\ell^2\cdot A\cdot(eh)e^2h\cdot(1/\ell)$. Here the names $HHC$ and $HEC$ distinguish $Y=\HH$ and $Y=\E$ in (ii). Summing yields
\begin{equation}\label{eq:q3closed}
q_3=p_{EHC}+p_{HHC}+p_{HEC}
=\frac{e^3h^2(1-eh+e^2h^2)}{\Delta}.
\end{equation}

\subsection{Pairs centred at a tree of height \texorpdfstring{$k-1$}{k-1}}\label{sec:pairs}
Let $J$ be a tree of exact height $k-1$, and put
\[
H^-=(J,E),\qquad H^+=(E,J).
\]
These are different ordered trees of height $k$. Although both have category $\HH$, their split neighbours are different. Consider a target forest containing
\begin{equation}\label{eq:pairtarget}
\LL\,P\,J\,E\,S\,\LL,
\end{equation}
where $P$ is nonempty. Let $u$ be the vertex marked at the last tree of $P$, and $v$ the vertex marked at $J$ in this target forest.

Suppose first that $P=Q\E$. Select the pair $\{u,v\}$ when all three of the following first-stage deletion tests hold:
\begin{equation}\label{eq:EJconditions}
\begin{array}{lll}
\text{(a)}&\text{the marked position in }P\underline{\HH}S
    &\text{is deleted},\\
\text{(b)}&\text{the marked position in }\underline{\E}S
    &\text{is deleted},\\
\text{(c)}&\text{the marked position in }Q\underline{\HH}\E S
    &\text{is deleted}.
\end{array}
\end{equation}
Each test means membership of the underlined position in $\mathcal D$ of the complete displayed word, with its two $\LL$ guards. In (a), the $\HH$ represents the root $H^-$ obtained by merging $J$ with its right-hand $E$; this is the merge neighbour of $v$. In (b), the word begins just after the $J\in\LL$ guard in the target forest; its first marker is the right horizontal neighbour of $v$. In (c), the $\HH$ represents $H^+$ obtained by merging the final $E$ of $P$ with $J$; this is the merge neighbour of $u$. In particular, these are three executions of the same rule on specified, generally different words, not three independent events.

For the second family, suppose $P=Q\E\HH$. Select $\{u,v\}$ when $\state(Q)\in\mathcal B$ and tests (a),(b) hold. The state condition says exactly that the penultimate $E$ of $P$ is deleted in the target's left segment. The two families are called the $(E,J)$ and $(H,J)$ pairs, according to the last tree of $P$.

\begin{proposition}\label{prop:pairs}
Every selected pair is present in $\Gh$, is incident to at most three undirected edges there, and has score at least two.
\end{proposition}
\begin{proof}
The vertex $v$ has its marker at $J\in\LL$, so it survives the first stage. The vertex $u$ is the final position before that $\LL$ guard, and also survives. Their common horizontal edge is therefore present.

At $v$, test (b) removes the right horizontal neighbour, and test (a) removes the merge neighbour. Besides the edge to $u$, only its split edge can remain. For an $(E,J)$ pair, $u$ has no split edge, and test (c) removes its merge neighbour, leaving at most its two horizontal edges. For an $(H,J)$ pair, the left horizontal neighbour of $u$ is removed by the state condition; $u$ cannot merge within the height bound, and its split neighbour is retained by Proposition~\ref{prop:horizontal}. The degree sum of the pair is thus at most four. Subtracting its one internal edge gives at most three incident edges and score at least $7\cdot2-4\cdot3=2$.
\end{proof}

\begin{figure}[htbp]
\centering
\begin{tikzpicture}[x=0.95cm,y=0.78cm,>=Stealth,
 every node/.style={font=\small,inner sep=3pt}]
\node at (0,0) {$\LL Q$}; \node at (1.8,0) {$\E$};
\node[draw,dashed] at (4.1,0) {$H^-=(J,E)$};
\node at (6.2,0) {$S\LL$};
\draw[->] (4.1,-.42)--(4.1,-1.12) node[midway,right] {split};
\node at (0,-1.6) {$\LL Q$};
\node[draw] at (1.8,-1.6) {$\E$};
\node[draw] at (3.5,-1.6) {$J$};
\node[draw,dashed] at (4.9,-1.6) {$\E$};
\node at (6.2,-1.6) {$S\LL$};
\draw[->] (2.65,-2.02)--(2.65,-2.72) node[midway,left] {merge};
\node at (0,-3.2) {$\LL Q$};
\node[draw,dashed] at (2.65,-3.2) {$H^+=(E,J)$};
\node at (4.9,-3.2) {$\E$};
\node at (6.2,-3.2) {$S\LL$};
\end{tikzpicture}
\caption{The $(E,J)$ pair with $P=Q\E$. The solid boxes are the two retained target vertices. The dashed markers, from top to bottom, are the neighbours tested in (a), (b) and (c) of \eqref{eq:EJconditions}. Each row describes a different forest except that both target markers belong to the middle row.}\label{fig:pair}
\end{figure}
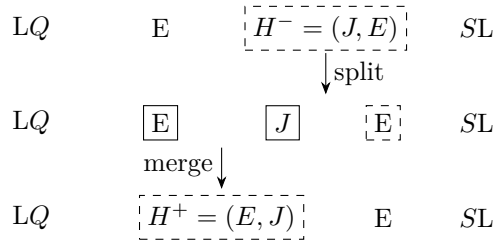

\Needspace{23\baselineskip}
\begin{example}[All three deletion tests for a pair]\label{ex:pair}
Take $Q=\E$, $P=\E\E$ and $S=\HH\E\HH$. With $k=3$, $J=(C,C)$ and $H=(J,J)$, an actual target forest is
\[
\mathcal F=(C,E,E,J,E,H,E,H,C).
\]
The target pair is $\{v(\mathcal F,3),v(\mathcal F,4)\}$. The three complete test words and their deletion sets are
\begin{center}
\begin{tabular}{clcc}
\toprule
Test&Word (underlined position tested)&Tested index&$\mathcal D(W)$\\
\midrule
(a)&$\E\E\underline{\HH}\HH\E\HH$&3&$\{2,3,4,5\}$\\
(b)&$\underline{\E}\HH\E\HH$&1&$\{1,2,3\}$\\
(c)&$\E\underline{\HH}\E\HH\E\HH$&2&$\{1,2,3,4,5\}$\\
\bottomrule
\end{tabular}
\end{center}
In (a) the underlined $\HH$ is $H^-=(J,E)$, and in (c) it is $H^+=(E,J)$; the other height-three trees stay unchanged. Every tested index belongs to its displayed deletion set. In the target itself, the left segment $\E\E$ is retained, while the right segment $\E\HH\E\HH$ loses its first three positions. Thus $v$ has lost its right and merge neighbours, and $u$ has lost its merge neighbour. For this choice of $J$, the split neighbour of $v$ is marked at $C$ and is retained. The two target degrees are exactly $(2,2)$, with one common edge, so this pair has three incident edges and score two.
\end{example}

\subsection{Common-suffix comparison}\label{sec:common-suffix}
The words in \eqref{eq:EJconditions} differ before $S$, but use the same $S$. To count simultaneous deletions, we need to compare their executions at the same point in this common suffix. The following statement makes the comparison precise. A candidate position and its exact difference have the meaning given in Section~\ref{sec:firstpassage}.

\begin{lemma}[Common-suffix comparison]\label{lem:common-suffix}
Consider finitely many executions of \eqref{eq:schedule}--\eqref{eq:selected-sets}. In each execution a specified position has already been processed, is still in the candidate, is not accepted, and has exact difference $d_j\ge-4$. Suppose that all executions have the same pending letter and are then continued with the same finite suffix, followed by the right $\LL$ guard. All specified positions are eventually accepted if and only if the specified position in an execution with minimum $d_j$ is eventually accepted. Conditions whose positions have already been accepted may be omitted; if any required position has been discarded, joint acceptance is impossible.
\end{lemma}
\begin{proof}
Until a specified position is accepted, its next score is determined by the common pending letter and the next common letter. Thus every as-yet-unaccepted execution receives the same increment. If the minimum difference is at least $-4$, each such execution continues its candidate by the first tie convention in \eqref{eq:selected-sets}. The order of their exact differences is preserved until one accepts. Any execution starting with a larger difference reaches a nonnegative trial difference no later than one starting with the minimum. Once it accepts, its specified position remains accepted and its condition can be removed.

It follows inductively that if the minimum-difference position accepts, all required positions have accepted by that time. Conversely, joint acceptance includes acceptance of that position. If its difference falls below $-4$ before acceptance, it cannot recover as an old candidate position; if the guard arrives first, it is not accepted. These failures are consistent with the same equivalence. No independence assumption is used.
\end{proof}

\subsection{Closed enumeration of the pair contexts}\label{sec:paircounts}
Let $w$ be the probability associated with $(-2,\HH)$ in \eqref{eq:passage}, and put $M=hw$. Define $T$ as the sum of $\wt(P)\wt(S)$ over all pairs of words $P,S$ satisfying the $(E,J)$ conditions, and define $T_c$ similarly for the $(H,J)$ conditions. These sums do not include the outer guards or the central $J,E$ in the target forest.

\begin{lemma}\label{lem:pairreduction}
The two context sums are
\begin{align}\label{eq:Tclosed}
T&=\frac{e(G_1+G_5)M+(G_3+G_5+G_8)r}{\ell},\\
T_c&=\frac{ehBM}{\ell}.\label{eq:Tcclosed}
\end{align}
Consequently the occurrence frequencies are
\begin{equation}\label{eq:q2closed}
p_{EJ}=\ell^2ejT,\qquad p_{HJ}=\ell^2ejT_c,\qquad q_2=\ell^2ej(T+T_c).
\end{equation}
\end{lemma}
\begin{proof}
First, condition (b) forces $S=\HH R$. If $S$ is empty, the initial $E$ in $\E S$ has terminal score $-1$ and is not accepted. If $S$ starts with $\E$, that position has candidate score $-5$ after processing and is discarded at the next step. With $S=\HH R$, reading the initial $\E\HH$ in test (b) leaves the tested $E$ in the candidate with exact difference $-1$ and pending letter $\HH$.

For an $(E,J)$ pair write $P=Q\E$. Run the recurrence separately on
\[
\underbrace{Q\E\underline{\HH}\HH}_{\text{test (a)}},\qquad
\underbrace{\underline{\E}\HH}_{\text{test (b)}},\qquad
\underbrace{Q\underline{\HH}\E\HH}_{\text{test (c)}}.
\]
The last $\HH$ in each displayed prefix is the \emph{first letter of $S$}. It has been read but not processed. All three underlined positions have now been processed. The unread letters are exactly the same word $R$ in all three executions, so Lemma~\ref{lem:common-suffix} applies at this point.

Table~\ref{tab:pair-states} gives the statuses at this precise stage for every possible $\state(Q)$. A numerical entry is the exact difference of the still-unaccepted candidate containing the tested position. The symbol $\mathsf A$ means that the tested position is already accepted, and $\mathsf X$ means that it cannot be accepted later. The column for (b) is included to show all three tests explicitly.
\begin{table}[htbp]
\centering
\caption{Pair tests after the first letter $\HH$ of $S=\HH R$ has been read. That letter is pending in every execution. The final column sums $\wt(S)=h\wt(R)$ over successful suffixes, without the right guard.}\label{tab:pair-states}
\begin{tabular}{cccccc}
\toprule
$\state(Q)$&$\state(Q\E)$&(a)&(b)&(c)&Suffix weight\\
\midrule
0&1&$-2$&$-1$&$\mathsf X$&$0$\\
1&1&$-2$&$-1$&$\mathsf A$&$M/\ell$\\
2&3&$-1$&$-1$&$\mathsf A$&$r/\ell$\\
3&5&$-1$&$-1$&$\mathsf A$&$r/\ell$\\
4&5&$-1$&$-1$&$-1$&$r/\ell$\\
5&1&$-2$&$-1$&$\mathsf A$&$M/\ell$\\
6&8&$-1$&$-1$&$\mathsf A$&$r/\ell$\\
7&1&$-2$&$-1$&$\mathsf X$&$0$\\
8&3&$-1$&$-1$&$\mathsf A$&$r/\ell$\\
\bottomrule
\end{tabular}
\end{table}

Here is a direct way to reproduce every row. Begin with the clipped state of $Q$. For (a), append $\E,\HH,\HH$ in that order, distinguishing the second of these letters. For (c), append $\HH,\E,\HH$, distinguishing the first. The numerical transitions follow \eqref{eq:state-update}. Once the distinguished letter is processed, also record whether it was accepted, remains in the candidate, or is discarded, using \eqref{eq:selected-sets} with the exact difference. For (b), the initial computation is always $\underline{\E}\HH$ and gives $-1$. This procedure uses only the displayed recurrences; representative prefixes for rows 0 through 8 are, respectively,
\[
\varepsilon,\ \E,\ \E\HH,\ \E\HH\E,\ \E\HH\HH,
\ \E\HH\E\E,\ \E\HH\E\HH,\ \E\HH\HH\HH,\ \E\HH\E\HH\E.
\]
The entries depend on the state, not on the choice of representative, because each tested position lies after $Q$ and is first included by the same state transition.

By the common-suffix comparison, rows 1 and 5 are governed by $(-2,\HH)$, with acceptance probability $w$. Including the first suffix letter and converting probability to a word-weight sum using \eqref{eq:suffix-normalization} gives $hw/\ell=M/\ell$. Rows 2, 3, 4, 6 and 8 are governed by $(-1,\HH)$, giving $hu/\ell=r/\ell$. Rows 0 and 7 contribute zero because test (c) has failed. Since $\wt(P)=e\wt(Q)$, the total is
\[
T=\frac{e(G_1+G_5)M+e(G_2+G_3+G_4+G_6+G_8)r}{\ell}.
\]
The identities $G_3=e(G_2+G_8)$, $G_5=e(G_3+G_4)$ and $G_8=eG_6$ reduce this to \eqref{eq:Tclosed}.

For an $(H,J)$ pair, $P=Q\E\HH$ with $\state(Q)\in\mathcal B$. Reading $P$ leaves state 6, and its penultimate $E$ has already been accepted. In test (a), after the additional root $\HH$ and the first suffix $\HH$ have been read, the tested root position has exact difference $-2$ and the suffix letter is pending. Test (b) again has difference $-1$ at the same pending letter. The successful suffix weight is therefore $M/\ell$. The allowed words $P$ have total weight $ehB$, proving \eqref{eq:Tcclosed}. Finally, the two outer guards and the central $J,E$ in \eqref{eq:pairtarget} contribute $\ell^2je$, proving \eqref{eq:q2closed}.
\end{proof}

\begin{example}[The continuation tie in row 4]\label{ex:continuation-tie}
Take $Q=\E\HH\HH$, which has state 4. At the comparison stage in the proof, the three prefixes are
\[
\E\HH\HH\E\underline{\HH}\HH,\qquad
\underline{\E}\HH,\qquad
\E\HH\HH\underline{\HH}\E\HH.
\]
For test (c), processing the underlined $\HH$ gives exact difference $-4$, with the following $\E$ pending. When the last displayed $\HH$ is read, the score of that $\E$ is $3$. Continuing the existing candidate and starting a new interval give equal scores, but only continuation keeps the underlined position. The first tie convention in \eqref{eq:selected-sets} therefore leaves that position in the candidate at difference $-1$. All three tests now have difference $-1$ and the same pending $\HH$.

For example, choose $R=\E\HH$, so $S=\HH\E\HH$. The remaining scores before the guard are $-1,3,-1$. Each tested position is accepted after the $3$, and the final $-1$ does not undo acceptance. This illustrates why the continuation tie at difference $-4$ is distinct from the acceptance tie at score zero in Example~\ref{ex:zero-score}.
\end{example}

\subsection{A separated triple}\label{sec:separated}
There is one further family, with target pattern
\begin{equation}\label{eq:isolated}
\LL\,\boxed{\E\ \HH\ \CC}\,\E\HH\E\HH\,\LL.
\end{equation}
Here there are no unspecified words between the displayed positions: the left segment is exactly $\E\HH$, and the right segment is exactly $\E\HH\E\HH$.

\begin{lemma}\label{lem:isolated}
The three boxed vertices in \eqref{eq:isolated} are present in $\Gh$, have degrees $(2,3,2)$, and are incident to five undirected edges. Their score is one, and their occurrence frequency is
\begin{equation}\label{eq:p0}
p_0=\ell^2e^5h^3.
\end{equation}
\end{lemma}
\begin{proof}
The left segment is retained. By Example~\ref{ex:basic-deletion}, the right segment loses its first three positions. Splitting $C$ produces the complete segment
\[
\E\HH\E_1\E_2\E\HH\E\HH.
\]
The split marker is at $E_1$, position 3 in this word. For completeness, its recurrence values are
\[
\begin{array}{c|rrrrrrrr}
i&1&2&3&4&5&6&7&8\\\hline
a_i&3&-1&-1&-1&3&-1&3&-1\\
I_i&-1&-2&-3&-4&-1&-2&1&0\\
O_i&0&0&0&0&0&0&1&1
\end{array}
\]
The continuation tie at position 5 preserves the interval beginning at position 1; position 7 accepts it. Thus $\mathcal D(W)=\{1,\ldots,7\}$ and the split neighbour is deleted.

The target $C$ has lost its split neighbour and its right horizontal neighbour. It retains its left horizontal neighbour and its merge neighbour marked at $(C,E)\in\LL$, of height two, so it has degree two. The first two target degrees are two and three by the same neighbour counts as in Proposition~\ref{prop:triples}. The two internal edges then give five incident edges. Multiplication of the two guard weights and the seven interior tree weights gives $\ell^2(eh)e^2(eheh)=\ell^2e^5h^3$.
\end{proof}

\subsection{Disjointness of the second-stage deletion sets}\label{sec:disjointness}
\begin{lemma}\label{lem:disjoint}
The target sets in Propositions~\ref{prop:triples}, \ref{prop:pairs} and Lemma~\ref{lem:isolated} are pairwise disjoint, and are disjoint from the first-stage deletion set.
\end{lemma}
\begin{proof}
It suffices to check positions in a fixed underlying forest, since vertices with different underlying forests are different. Every triple is specified by its centre $C$ and occupies that position and the two positions immediately to its left. Those preceding trees both have categories in $\{\E,\HH\}$. Distinct eligible $C$ centres cannot be one or two positions apart, so their triples do not overlap. At a fixed centre, cases (i) and (ii) of Proposition~\ref{prop:triples} have different first boxed letters. The separated triple has an $E$ immediately to the right of $C$, whereas the other triple families have $\HH$ there.

Every pair is specified by its centre $J$ and occupies that position and the position immediately to its left, which has category $\E$ or $\HH$. Thus distinct eligible pair centres cannot be adjacent. At a fixed centre the two pair families have different preceding letters. A pair centre cannot be a triple position: its category is $\LL$, but it is not $C$ because $k\ge3$ and $J$ has height $k-1$. The left position of a pair cannot be one of the first two positions of a triple, because its right neighbour would then have category $\E$ or $\HH$, or be the tree $C$, whereas a pair requires that neighbour to be $J$. Nor can the left position of a pair be the triple centre, since its category is not $\LL$. This proves pair--triple disjointness.

The preceding propositions and lemma proved that every target vertex survives the first stage. This completes the proof.
\end{proof}

Disjoint vertex sets can still have edges between them. Accordingly, the individual incident-edge counts give an upper bound for the edge loss under simultaneous deletion, not an equality. Section~\ref{sec:counts} uses this inequality to obtain the density lower bound.

\section{Evaluation and passage to finite subgraphs}\label{sec:counts}
With $s=\sqrt3$ and $j=(s-\sqrt{2s-1})/2$, the quantities in \eqref{eq:Dmain} are
\begin{align}\label{eq:mainconstants}
\omega&=\frac{120197s-207644}{455008},&
\rho&=\frac{1410649978-795527183s}{6469758752},\notag\\
q_3&=\frac{24789s-42440}{1820032},&
q_2&=\frac{140462279622-81074695963s}{414064560128}\,j,\\
p_0&=\frac{518-299s}{131072}.&&\notag
\end{align}
The first triple and pair families, without the separated triple, give
\begin{equation}\label{eq:Dfour}
D_4=\frac72+\frac{\omega+q_3+2q_2}{2(1-\rho-3q_3-2q_2)},\qquad
3.5007448434200<D_4<3.5007448434201.
\end{equation}

Equations~\eqref{eq:Gclosed}--\eqref{eq:rhoformula}, \eqref{eq:q3closed}, \eqref{eq:Tclosed}--\eqref{eq:q2closed} and \eqref{eq:p0} evaluate every term in Theorem~\ref{thm:main} by rational operations on $e,h,\ell,j$. For reference, individual values are
\begin{align}\label{eq:individual}
p_{EHC}&=\frac{264737-152284\sqrt3}{3640064},\notag\\
p_{HHC}&=\frac{-154158+89005\sqrt3}{1820032},\qquad
p_{HEC}=\frac{-41301+23852\sqrt3}{3640064},\\
T&=\frac{314179000362-180976553799\sqrt3}{782840808992},\notag\\
T_c&=\frac{16841696349-9723292976\sqrt3}{24463775281}.\notag
\end{align}
These equalities follow by substitution into the formulas derived in Sections~\ref{sec:horizontal} and~\ref{sec:gadgets}.

\begin{proof}[Proof of Theorem~\ref{thm:main}]
For fixed $k\ge3$ use the parameters $e_k,h_k,\ell_k$ of \eqref{eq:finiteweights} and
$j_k=\Phi_{k-1}(\xi_k)-\Phi_{k-2}(\xi_k)$ in the same formulas. Denote the resulting quantities by a subscript $k$. The positive denominators and Lemma~\ref{lem:segments} give convergence of all these quantities to \eqref{eq:mainconstants} as $k\to\infty$.

Let $Y_{n,k}$ be the vertices remaining after the horizontal selection and all the local deletions above. By Lemma~\ref{lem:disjoint}, its limiting fraction among $\BB(n,k)$ is
\[
\sigma_k=1-\rho_k-3q_{3,k}-3p_{0,k}-2q_{2,k}.
\]
The horizontal stage loses $(7\rho_k-\omega_k)/2$ in doubled edge count per original vertex. Each triple reduces the doubled edge count by at most ten, and each pair by at most six. Edges between different deletion sets can be counted more than once by these incident-edge estimates; this only strengthens the lower bound. Therefore
\begin{equation}\label{eq:Dk}
\liminf_{n\to\infty}\dens(Y_{n,k})\ge
\frac72+\frac{1-4\xi_k+\omega_k+q_{3,k}+p_{0,k}+2q_{2,k}}{2\sigma_k}.
\end{equation}
Its limit as $k\to\infty$ is \eqref{eq:Dmain}. The rational comparison in Appendix~\ref{app:comparison} proves $\sigma=\lim_k\sigma_k>0.9941$ and the displayed enclosure for $D_+$. In particular the right side of \eqref{eq:Dk} is greater than $3.50074529$ for all sufficiently large $k$, and then $\dens(Y_{n,k})>3.50074529$ for all sufficiently large $n$. Equation~\eqref{eq:cheeger} gives the boundary bound.
\end{proof}

\section{Exact optimization on a prescribed window}\label{sec:windows}
The product law also permits optimization over every retention decision on a fixed window. We state this for the coarse alphabet $\mathcal A=\{\E,\LL,\HH\}$. The required distribution records the categories at both endpoints of a split edge, not only the category of a single tree.

Put
\[
u_1=\sqrt3/2,\qquad u_2=\sqrt{2\sqrt3-1}/2,
\quad p=(e,\ell,h),\quad v_1=(e,u_1-e,0),\quad v_2=(e,u_2-e,0).
\]
Let $\kappa_{tij}$ denote the mass of a tree in category $t$ whose left and right children have categories $i,j$. Then
\begin{equation}\label{eq:kappa}
\kappa_{\E ij}=0,\qquad
\kappa_{\LL ij}=(v_2)_i(v_2)_j,\qquad
\kappa_{\HH ij}=(v_1)_i(v_1)_j-(v_2)_i(v_2)_j.
\end{equation}
Indeed a root of height at most $k-1$ has children of height at most $k-2$, while a height-$k$ root is obtained by subtracting that case from roots with children of height at most $k-1$. Evaluating the child generating polynomials at $\xi_k$ and passing to the limit gives \eqref{eq:kappa}. These nonnegative masses sum to $\ell$ and $h$ in the two nonzero rows, and their total is $3/4$.

Fix $a,b\ge0$ and $w=a+b+1$. The type of a marker is the word of $w$ categories with $a$ positions to its left and $b$ to its right. A retention rule is a nonempty set $U\subseteq\mathcal A^w$; retain precisely its types. Forest-end conventions have zero limiting effect. In this section $\wt(c_1\cdots c_r)=\prod_i p_{c_i}$ for letters in $\mathcal A$. For a type $v$, put $\mu_v=\wt(v)$.

For a word $W$ of length at least $r$, let $\pref_r(W)$ denote its first $r$ letters, and let $\pref_0(W)$ be the empty word. Construct a weighted undirected graph of types. For each $z\in\mathcal A^{w+1}$ put weight $\wt(z)$ between its first and last $w$ letters; this counts a right marker move. For each $A\in\mathcal A^a$, $B\in\mathcal A^b$ and each $t,i,j$, put weight
\begin{equation}\label{eq:windowedge}
\wt(AB)\kappa_{tij}
\quad\text{between}\quad AtB\quad\text{and}\quad A\pref_{b+1}(ijB).
\end{equation}
This counts a split. Aggregate weights with the same unordered pair of endpoint types as $m_{uv}$, for a fixed ordering $u\le v$. A loop in the graph of types denotes equal \emph{types} at distinct Cayley vertices.

\begin{theorem}[Exact finite-window optimization]\label{thm:window}
The limiting density of the rule $U$ is
\begin{equation}\label{eq:windowdensity}
D(U)=\frac{2\sum_{u\le v,\ u,v\in U}m_{uv}}{\sum_{u\in U}\mu_u}.
\end{equation}
Its maximum over all nonempty $U$ is the optimum of
\begin{align}\label{eq:windowLP}
\text{maximize }&2\sum_{u<v}m_{uv}z_{uv}+2\sum_u m_{uu}x_u,\notag\\
\text{subject to }&0\le z_{uv}\le x_u,\quad z_{uv}\le x_v,\quad
x_u\ge0,\quad\sum_u\mu_u x_u=1.
\end{align}
All coefficients and the optimum belong to $K=\Q(\sqrt3,\sqrt{2\sqrt3-1})$.
An upper bound $d$ has a certificate consisting of allocations $A_{uv}$, $u<v$, satisfying
\begin{align}\label{eq:allocation}
0\le A_{uv}&\le2m_{uv},\notag\\
2m_{uu}+\sum_{v>u}A_{uv}+\sum_{v<u}(2m_{vu}-A_{vu})&\le d\mu_u
\quad\text{for every }u.
\end{align}
At the optimum, such an allocation exists in $K$.
\end{theorem}
\begin{proof}
Lemma~\ref{lem:ratio} counts vertex and horizontal endpoint masses. Equation~\eqref{eq:kappa} counts the root and both children jointly, and the marker stays at the left child, giving \eqref{eq:windowedge}. Each undirected Cayley edge is counted once, in the rightward or split direction. Requiring both types to survive gives \eqref{eq:windowdensity}. As a check, $\sum_u\mu_u=1$ and $\sum_{u\le v}m_{uv}=1+3/4=7/4$.

For a set $U$ of mass $M$, take $x_u=\ind_{u\in U}/M$ and $z_{uv}=\min(x_u,x_v)$ in \eqref{eq:windowLP}. Conversely, for any feasible solution put $U_t=\{u:x_u>t\}$. The identities
\[
x_u=\int_0^\infty\ind_{u\in U_t}\,dt,\qquad
\min(x_u,x_v)=\int_0^\infty\ind_{u,v\in U_t}\,dt
\]
show that its objective is at most a weighted average of the densities $D(U_t)$, because the integrated vertex mass equals one. Some threshold set is at least as good. Thus there is no gap between the linear program and the set optimum. This is the weighted form of the classical densest-subgraph method \cite{Goldberg}.

To prove the upper certificate directly, sum \eqref{eq:allocation} over $u\in U$. Each retained internal edge contributes its full doubled weight, and edges with only one retained endpoint contribute a nonnegative amount. This proves $D(U)\le d$. The dual of \eqref{eq:windowLP} assigns the doubled weight of each off-diagonal edge between its endpoints and gives exactly these constraints; excess allocation can be reduced. Finite-dimensional linear-programming duality gives equality at the optimum. Since the coefficients are in $K$, an optimal basic solution of the primal and dual is obtained by linear algebra over $K$.
\end{proof}

\subsection{A short certificate for arbitrary left context}
For windows with one position to the right of the marker, we can give an explicit certificate valid for every fixed number of positions to the left.
\begin{theorem}[One right neighbour]\label{thm:one-right}
For every fixed $a\ge0$, among all coarse-category retention rules on the window $[-a,1]$, the maximum limiting density is $7/2$. The limit is taken first as $n\to\infty$ and then as $k\to\infty$, as above.
\end{theorem}
\begin{proof}
Set
\[
z=u_2=\sqrt\ell,\qquad \theta=\frac{z-e}{\ell},\qquad
(t_\E,t_\LL,t_\HH)=(1,\theta,0),\qquad \eta_b=\ind_{b=\HH}.
\]
Thus $\sum_b p_bt_b=z$, $\sum_b p_b\eta_b=h$, and $h+z^2=3/4$.
Define a horizontal bias, indexed by two categories, by
\begin{equation}\label{eq:onebias}
f_{bc}=2\eta_b+2\eta_c+2t_c(t_b+z)-3.
\end{equation}
In the order $\E,\LL,\HH$ its nine entries are
\begin{equation}\label{eq:biasmatrix}
(f_{bc})=
\begin{pmatrix}
2z-1&2\theta(1+z)-3&-1\\
2(\theta+z)-3&2\theta(\theta+z)-3&-1\\
2z-1&2\theta z-1&1
\end{pmatrix}.
\end{equation}
All lie in $[-1,1]$: the elementary bounds $3/4<z<4/5$ and $4/5<\theta<9/10$ suffice for each displayed entry.

We allocate the doubled mass of every elementary edge before aggregating equal endpoint types. On a rightward marker edge, let $c,d$ be the marked category and its right neighbour \emph{at the target}. Assign $1+f_{cd}$ times the edge mass to its source, and $1-f_{cd}$ times the mass to its target. For a split with parent in $\HH$, assign all doubled mass to the parent endpoint. For a split with parent in $\LL$, assign it all to the left-child endpoint. Every allocation is nonnegative and uses exactly the doubled edge mass.

Consider a type ending in $b,c$, of mass $\mu$. The total allocated split mass at this type, divided by $\mu$, is
\[
2\eta_b+2t_bt_c.
\]
The first term is the contribution of its own height-$k$ root; the second is the joint mass $\kappa_{\LL bc}=(v_2)_b(v_2)_c$, divided by $p_bp_c$. The $a$ preceding categories have the same product weight on both split endpoints and cancel from this quotient.

Put $g_c=\sum_d p_df_{cd}$. Equation~\eqref{eq:onebias} and $h+z^2=3/4$ give
\[
g_c=2\eta_c+2zt_c-\frac32,\qquad
f_{bc}=2\eta_b+2t_bt_c-\frac32+g_c.
\]
The rightward edge allocation contributes $1+g_c$ per unit vertex mass; incoming rightward edges contribute $1-f_{bc}$. Including the split allocation, the total is consequently
\[
2+g_c-f_{bc}+2\eta_b+2t_bt_c=\frac72.
\]
This calculation is independent of $a$. Summing the nonnegative allocations over any retained collection of types proves the upper bound as in \eqref{eq:allocation}. Retaining all types attains $7/2$, proving equality.
\end{proof}

For example, $a=1$ gives the exact optimum for every rule on the three positions $[-1,1]$; the same nine entries prove the result for every larger fixed left context. The statement does not enlarge the right window, refine the category alphabet, or allow root inspections. The constructions of Sections~\ref{sec:gadgets}--\ref{sec:counts} distinguish the ordered children of a root and use whole segments rather than a fixed category window. Their densities can therefore exceed this restricted optimum.

\section*{Acknowledgements}
The author thanks Matthew Brin for comments that improved the contextual discussion in the introduction. Generative-AI tools, principally OpenAI ChatGPT and Anthropic Claude, were used as research aids for exploratory calculations, development and testing of candidate deletion rules, code generation and review, symbolic and exact-arithmetic checks, and drafting and revision of the manuscript. The author independently checked the mathematical claims and takes full responsibility for the contents.

\appendix
\section{An elementary rational comparison}\label{app:comparison}
Set $s=\sqrt3$ and $t=\sqrt{2s-1}$. The following rational intervals suffice; every terminating decimal in this appendix denotes an exact rational number:
\begin{align*}
1.732050807568877293527446341505&<s<1.732050807568877293527446341506,\\
1.569745716712663811642153567306&<t<1.569745716712663811642153567307.
\end{align*}
The first pair is checked by squaring. For the second, square the lower endpoint and compare it with $2s_--1$, then square the upper endpoint and compare it with $2s_+-1$, where $s_-,s_+$ are the first pair of rational endpoints. Propagate these intervals through $j=(s-t)/2$ and \eqref{eq:mainconstants}. Ordinary rational interval arithmetic, taking account of coefficient signs, gives
\[
0.9941<1-\rho-3q_3-3p_0-2q_2<0.9942,
\]
and
\[
3.5007448434200<D_4<3.5007448434201,\qquad
3.5007452936557<D_+<3.5007452936559.
\]
This is a finite rational calculation using only the five constants in \eqref{eq:mainconstants}; it uses no enumerated state list or floating-point assertion.

\end{document}